\documentclass{amsart}
\usepackage{hyperref}
\usepackage{enumerate} 
\usepackage{upref}
\usepackage{verbatim} 
\usepackage{color}
\usepackage{graphicx}
\usepackage{bbm}
\usepackage[textsize=tiny]{todonotes}

\newcommand*{\mailto}[1]{\href{mailto:#1}{\nolinkurl{#1}}}

\usepackage[latin1]{inputenc}
\usepackage{amssymb, amsmath, amsfonts, amsthm}
\usepackage[all]{xy}

\usepackage{marginnote}

\DeclareMathOperator{\id}{Id}

 \DeclareMathOperator{\arcsinh}{arcsinh}
\newcommand{\dott}{\, \cdot\,}

\newcommand{\Gr}{G}
\newcommand{\U}{\ensuremath{\mathcal{U}}}

\newcommand{\D}{\ensuremath{\mathcal{D}}}

\newcommand{\F}{\ensuremath{\mathcal{F}}}

\newcommand{\inv}{{^{-1}}}
\newcommand{\abs}[1]{\left\vert#1\right\vert}

\newcommand{\Real}{\mathbb R}

\newcommand{\indicator}{\mathbb I}
\newcommand{\norm}[1]{\left\Vert#1\right\Vert}

\newcommand{\muac}{\mu_{\text{\rm ac}}}

\DeclareMathOperator{\erf}{erf}

\newtheorem{theorem}{Theorem}[section]

\newtheorem{definition}[theorem]{Definition}

\newtheorem{example}[theorem]{Example}
\newtheorem{remark}[theorem]{Remark}

\numberwithin{equation}{section}

\allowdisplaybreaks

\begin{document}

\title[Metric for the HS equation]{A Lipschitz metric for the Hunter--Saxton equation}

\author[J. A. Carrillo]{Jos\'e Antonio Carrillo}
\address{Department of Mathematics\\ Imperial College London \\ South Kensington Campus\\ London SW7 2AZ\\ UK}
\email{\mailto{carrillo@imperial.ac.uk}}
\urladdr{\url{http://wwwf.imperial.ac.uk/~jcarrill/index.html}}

\author[K. Grunert]{Katrin Grunert}
\address{Department of Mathematical Sciences\\ NTNU Norwegian University of Science and Technology\\ NO-7491 Trondheim\\ Norway}
\email{\mailto{katring@math.ntnu.no}}
\urladdr{\url{http://www.math.ntnu.no/~katring/}}

\author[H. Holden]{Helge Holden}
\address{Department of Mathematical Sciences\\
  NTNU Norwegian University of Science and Technology\\
  NO-7491 Trondheim\\ Norway}
\email{\mailto{holden@math.ntnu.no}}
\urladdr{\url{http://www.math.ntnu.no/~holden/}}

\thanks{JAC was partially supported by the Royal Society via a Wolfson Research Merit Award. HH and KG acknowledge the support by the grant {\it Waves and Nonlinear Phenomena (WaNP)} from the Research Council of Norway. This research was done while the authors were at Institut Mittag-Leffler, Stockholm.}  
\subjclass[2010]{Primary: 35Q53, 35B35; Secondary: 35B60}
\keywords{Hunter--Saxton equation, Lipschitz metric, conservative solution}

\begin{abstract}
We analyze stability of conservative solutions of the Cauchy problem on the line for the (integrated) Hunter--Saxton (HS) equation.  Generically, the solutions of the HS equation develop singularities with steep gradients while preserving continuity of the solution itself. In order to obtain uniqueness, one is required to augment the equation itself by a measure that represents the associated energy, and the breakdown of the solution is associated with a complicated interplay where the measure becomes singular. 
The main result in this paper is the construction of a Lipschitz metric that compares two solutions of the HS equation with the respective initial data. The Lipschitz metric is based on the use of the Wasserstein metric.
\end{abstract}
\maketitle

\section{Introduction}

In this paper we consider the Cauchy problem for conservative solutions of the (integrated) Hunter--Saxton (HS) equation \cite{MR1135995}
\begin{equation}\label{eq:HSa}
u_t+uu_x = \frac14 \int_{-\infty}^x u_x^2(y)dy-\frac14\int_x^\infty u_x^2(y)dy, \quad u|_{t=0}=u_0.
\end{equation}
The equation has been extensively studied, starting with \cite{MR1361013, MR1361014}.   The initial value problem is not well-posed without further constraints: Consider the trivial case $u_0=0$ which clearly has as one solution $u(t,x)=0$. However, as can be easily verified, also 
\begin{equation}\label{eq:counter}
u(t,x)=-\frac{\alpha}4 t\, \indicator_{(-\infty, -\frac{\alpha}8 t^2)}(x)+ \frac{2x}{t}\, \indicator_{(-\frac{\alpha}8 t^2, \frac{\alpha}8 t^2)}(x)+ \frac{\alpha}4 t\, \indicator_{(\frac{\alpha}8 t^2, \infty)}(x)
\end{equation}
is a solution for any $\alpha\ge0$. Here  $\indicator_A$ is the indicator (characteristic) function of the set $A$.

Furthermore, it turns out that the solution $u$ of the HS equation may develop singularities in finite time in the following sense:  Unless the initial
data is monotone increasing, we find
\begin{equation}\label{eq:blow-up}
  \inf(u_x)\to-\infty \text{  as  } t\uparrow t^*=2/\sup(-u_0^\prime).
\end{equation}
Past wave breaking there are at least two different classes of solutions, denoted conservative (energy is conserved) and dissipative (where energy is removed locally) solutions, respectively, and this dichotomy is the source of the interesting behavior of solutions of the equation.  We will in this paper consider the so-called conservative case where an associated energy is preserved. 

Zhang and Zheng \cite{MR1668954,MR1701136,MR1799274} gave the first proof of global solutions of the HS equation on the half-line using Young measures and mollifications with compactly supported initial data. Their proof covered both the conservative case and the dissipative case.  Subsequently, Bressan and Constantin \cite{MR2191785}, using a clever rewrite of the equation in terms of new variables, showed global existence of conservative solutions without the assumption of compactly supported initial data.  The novel variables turned the partial differential equation into a system of linear ordinary differential equations taking values in a Banach space, and where the singularities were removed. 
A similar, but considerably more complicated, transformation can be used to study the very closely related  Camassa--Holm equation, see \cite{MR2278406,HolRay:07}. 
The convergence of a numerical method to compute the solution of the HS equation can be found in \cite{HolKarRis:sub05}.

We note in passing that the original form of the HS equation is
\begin{equation*}
(u_t+uu_x)_x=\frac12 u_x^2,
\end{equation*}
and like most other researchers working on the HS equation, we prefer to work with an integrated version. However, in addition to \eqref{eq:HSa}, one may study, for instance,
\begin{equation*}
u_t+uu_x=\frac12\int_0^x u_x^2(y)dy,
\end{equation*}
and while the properties are mostly the same, the explicit solutions differ. 

Our aim here is to determine a Lipschitz metric $d$ that compares two solutions $u_1(t),u_2(t)$ at time $t$ with the corresponding initial data, i.e., 
\begin{equation*}
d(u_1(t),u_2(t))\le C(t)  d(u_1(0),u_2(0)),             
\end{equation*}
where $C(t)$ denotes some increasing function of time.  The existence of such a metric is clearly  intrinsically connected with the uniqueness question, and as we could see from  the  example where \eqref{eq:counter} as well as the trivial solution both satisfy the equation, this is not a trivial matter.  Unfortunately, none of the standard norms in $H^s$ or $L^p$ will work.  A Lipschitz metric was derived in \cite{BHR}, and we here offer an alternative metric that also provides a simpler and more efficient way to solve the initial value problem.   

Let us be now more precise about the notion of solution. We consider the Cauchy problem for the integrated and augmented HS equation, which, in the conservative case, is given by 
\begin{subequations}\label{eq:HS}
\begin{align}\label{HS:1}
u_t+uu_x& = \frac14 \int_{-\infty}^x d\mu-\frac14\int_x^\infty d\mu,\\ \label{HS:2}
\mu_t+(u\mu)_x& =0.
\end{align}
\end{subequations}
In order to study conservative solution, the HS equation \eqref{HS:1} is augmented by the second equation \eqref{HS:2} that keeps track of the energy. 
A short computation reveals that if the solution $u$ is smooth and $\mu=u_x^2$, then the equation \eqref{HS:2} is clearly satisfied.  In particular, it shows that the energy $\mu(t,\Real)=\mu(0,\Real)$ is constant in time.  However, the challenge is to treat the case without this regularity, and the proper way to do that is to let $\mu$ be a nonnegative and finite Radon measure.   When there is a blow-up in the spatial derivative of the solution (cf. \eqref{eq:blow-up}), energy is transferred from the absolutely continuous part of the measure to the singular part, and, after the blow-up, the energy is transferred back to the absolutely continuous part of the measure. Thus, we will consider the solution space consisting of all pairs $(u,\mu)$ such that 
\begin{equation*}
u(t,\dott)\in L^\infty(\Real), \quad u_x(t,\dott)\in L^2(\Real), \quad \mu(t,\dott)\in \mathcal{M}_+(\Real) \quad \text{ and }\quad d\muac=u_x^2 dx,
\end{equation*}
where $\mathcal{M}_+(\Real)$ denotes the set of all nonnegative  and finite Radon measures on $\Real$. 

We would like to identify a natural Lipschitz metric, which measures the distance between pairs $(u_i,\mu_i)$, $i=1,2$, of solutions. The Lipschitz metric constructed in \cite{BHR} (and extended to the two-component HS equation in \cite{anders,andersPhD}) is based on the reformulation of the HS equation in Lagrangian coordinates which at the same time linearizes the equation. However, there is an intrinsic non-uniqueness in  Lagrangian coordinates as there are several distinct ways to parametrize the particle trajectories for one and the same solution in the original, or Eulerian, coordinates. This has to be accounted for when one measures the distance between solutions in Lagrangian coordinates, as one has to identify different elements belonging to one and the same equivalence class. We denote this as relabeling. In addition, for this construction one not only needs to know the solution in Eulerian coordinates, but also in Lagrangian coordinates for all $t$.

The present approach is based on the fact that a natural metric for measuring distances between Radon measures (with the same total mass) is given through the Wasserstein (or Monge--Kantorovich) distance $d_W$, which in one dimension is defined with the help of pseudo inverses, see \cite{Vil03}. This tool has been used extensively in the field of kinetic equations \cite{LT,CTPE}, conservation laws \cite{MR2134955,CDL} and nonlinear diffusion equations \cite{CT03,CGT03,CDG06}. To be more precise, given two positive and finite Radon measures $\mu_1$ and $\mu_2$, where we for simplicity assume that 
$\mu_1(\Real)=\mu_2(\Real)=C$, let 
\begin{equation} \label{eq:1}
F_i(x)=\mu_i((-\infty,x)) \quad i=1,2,
\end{equation}
and define their pseudo inverses $\chi_i\colon[0,C]\to \Real$ as follows
\begin{equation*}
\chi_i(\xi)=\sup\{x\mid F_i(x)<\xi\}.
\end{equation*}
Then, we define 
\begin{equation*}
d_W(\mu_1,\mu_2)= \norm{\chi_1-\chi_2}_{L^1([0,C])}.
\end{equation*}
As far as the distance between $u_1$ and $u_2$ is concerned, we are only interested in measuring the ``distance in the $L^\infty$ norm". Thus we introduce the distance $d$ as follows 
\begin{equation*}
d((u_1,\mu_1),(u_2,\mu_2))= \norm{u_1(\chi_1(\dott))-u_2(\chi_2(\dott))}_{L^\infty([0,C])}+d_W(\mu_1,\mu_2).
\end{equation*}
For this to work, it is necessary that this metric behaves nicely with the time evolution. Thus as a first step, we are interested in determining the time evolution of both $\chi(t,x)$, the pseudo inverse of $\mu(t,x)$, and $u(t,\chi(t,x))$. 

Let $(u(t), \mu(t))$ be a weak conservative solution to the HS equation with total energy $\mu(t,\Real)=C$.
To begin with, we assume that $F(t,x)$ is \textit{strictly increasing and smooth}, which greatly simplifies the analysis. Recall that $\chi(t,\dott)\colon[0,C]\to \Real$ is given by 
\begin{equation*} 
\chi(t,\eta)=\sup\{x\mid\mu(t,(-\infty,x))<\eta\}=\sup\{x\mid F(t,x)<\eta\}.
\end{equation*}
According to the assumptions on $F(t,x)$, we have that $F(t,\chi(t,\eta))=\eta$ for all $\eta\in [0,C]$ and $\chi(t,F(t,x))=x$ for all $x\in\Real$. Direct formal calculations yield that 
\begin{subequations}\label{HS:34}
\begin{align}
\chi_t(t, F(t,x))+\chi_\eta(t, F(t,x))F_t(t,x)& =0,\\
\chi_\eta(t,F(t,x))F_x(t,x)&=1.
\end{align}
\end{subequations}
Recalling \eqref{HS:2} and the definition of $F(t,x)$, we have 
\begin{subequations}\label{HS:5}
\begin{align}
F_x(t,x)&=\mu(t,x), \\
F_t(t,x)&=\int_{-\infty}^xd\mu_t(t)=-\int_{-\infty}^xd(u(t)\mu(t))_x=  -u(t,x)\mu(t,x).
\end{align}
\end{subequations}
Thus combining \eqref{HS:34} and \eqref{HS:5}, we obtain 
\begin{align*}
\chi_t(t,F(t,x))&=-\chi_\eta(t, F(t,x))F_t(t,x)= \chi_\eta(t, F(t,x))u(t,x)\mu(t,x)\\
&= \chi_\eta(t, F(t,x)) F_x(t,x) u(t,x)=u(t,x).
\end{align*}
Introducing $\eta=F(t,x)$, we end up with
\begin{equation*}
\chi_t(t,\eta)= u(t,\chi(t,\eta)),
\end{equation*}
where we again have used that $\chi(t,F(t,x))=x$ for all $x\in\Real$. 
As far as the time evolution of $\U(t,\eta)=u(t,\chi(t,\eta))$ is concerned, we have 
\begin{align*}
\U_t(t,\eta)& = u_t(t,\chi(t,\eta))+u_x(t,\chi(t,\eta))\chi_t(t,\eta)\\
&= u_t(t,\chi(t,\eta))+uu_x(t,\chi(t,\eta)) \notag\\ 
& = \frac14 \int_{-\infty}^{\chi(t,\eta)}d\mu(t,\sigma)-\frac14 \int_{\chi(t,\eta)}^\infty d\mu(t,\sigma)\notag\\ 
& = \frac12 \int_{-\infty}^{\chi(t,\eta)}d\mu(t,\sigma)-\frac14 C\notag\\
& = \frac12 F(t,\chi(t,\eta))-\frac14 C\notag\\
& =\frac12 \eta-\frac14 C.\notag
\end{align*}
Thus we get the very simple system of ordinary differential equations
\begin{subequations}\label{HS:ode}
\begin{align}
\chi_t(t,\eta)&= \U(t,\eta), \\
\U_t(t,\eta)&  =\frac12 \eta-\frac14 C.
\end{align}
\end{subequations}
The global solution of the initial value problem is simply given by 
\begin{equation}\label{eq:10}
\{(\chi(t,\eta),t, \U(t,\eta))\in\Real^3 \mid t\in (0,\infty), \quad \eta\in [0,C]\}.
\end{equation}
The above derivation is only of formal character,  and this derivation is but valid if $F(t,x)$ is strictly increasing and smooth. 
However, it turns out that the simple result \eqref{HS:ode} also persists in the general case, but the proof is considerably more difficult, and is the main result of this paper. 

\medskip
We prove two results. The first result, Theorem  \ref{thm:existence}, describes a simple and explicit formula for conservative solutions of the Cauchy problem.
Let $u_0\in H^1(\Real)$ and $\mu_0$ be a nonnegative, finite Radon measure with $C=\mu_0(\Real)$.   Define
\begin{subequations}\label{HS:ini0}
\begin{align}
\chi_0(\eta)&=\sup\{x\mid \mu_0((-\infty,x))<\eta\}, \\
\U_0(\eta)&=u_0(\chi_0(\eta)).
\end{align}
\end{subequations}
If $\lim_{\eta\to 0}\chi_0(\eta)=-\lim_{\eta\to C}\chi_0(\eta)=-\infty$ (all other cases are treated in Theorem~\ref{thm:existence}), we define
\begin{subequations}\label{HS:ini01}
\begin{align}
\chi(t,\eta)&=
\frac{t^2}{4}(\eta-\frac{C}2)+t\,\U_0(\eta)+\chi_0(\eta), \quad \text{if $\eta\in(0,C)$}, \\  
\U(t,\xi)&=
\frac{t}{2}(\eta-\frac{C}2)+\U_0(\eta), \quad \text{if  $\eta\in(0,C)$}.
\end{align}
\end{subequations}
Then we have
\begin{multline*}
\{(x,t,u(t,x))\in \Real^3\mid t\in[0,\infty),\, x\in\Real \} \\
   =\{(\chi(t,\eta),t,\U(t,\eta))\in \Real^3\mid t\in[0,\infty),\, \eta\in(0,C) \},
\end{multline*}
where $u=u(t,x)$ denotes the conservative solution of the HS equation \eqref{eq:HS}. 

The second result, Theorem \ref{thm:main}, describes the Lipschitz metric. Let $u_{0,j}\in H^1(\Real)$ 
and $\mu_{0,j}$ be a nonnegative, finite Radon measure with $C_j=\mu_{0,j}(\Real)$ for $j=1,2$, and define
$\chi_j(t,\eta)$ and $\U_j(t,\eta)$ by \eqref{HS:ini0} and \eqref{HS:ini01} for $j=1,2$ where $u_0$ is replaced by $u_{0,j}$ and 
$\mu_{0}$ is replaced by $\mu_{0,j}$, respectively.  Next introduce
$\hat\chi_j(t,\eta)=\chi_j(t,C_j\eta) $ and $\hat \U_j(t,\eta)=\U_j(t,C_j\eta)$ for $j=1,2$.
Define
\begin{align*}
d((u_1(t),\mu_1(t)), &(u_2(t),\mu_2(t))) \\
&\qquad = \norm{\hat\U_1(t,\dott)-\hat\U_2(t,\dott)}_{L^\infty([0,1])}\\
&\qquad\quad+\norm{\hat\chi_1(t,\dott)-\hat\chi_2(t,\dott)}_{L^1([0,1])}
+\abs{C_1-C_2}.
\end{align*}
Then we have, see Theorem \ref{thm:main}, that
\begin{align*}
d((u_1(t)&,\mu_1(t)) ,(u_2(t),\mu_2(t)))\\
& \leq (1+t+\frac18 t^2)d((u_1(0),\mu_1(0)),(u_2(0),\mu_2(0))).
\end{align*}


\section{The Lipschitz metric for the Hunter--Saxton equation}

Let us study the calculations \eqref{eq:1}--\eqref{eq:10} on two explicit examples.
 \begin{example}\label{rem:smooth}  
(i) Let
\begin{equation*}
u_0(x)=\Big(\frac{\pi}{2}\Big)^{1/2}\erf(\frac{x}{\sqrt{2}}), \quad \mu_0(x)=u_{0,x}^2(x)dx=e^{-x^2} dx,
\end{equation*}
where $\erf(x)=\frac{2}{\sqrt{\pi}}\int_0^x e^{-t^2}dt$ is the error function.
We find that 
\begin{equation*}
F_0(x)=\mu_0((-\infty,x))= \frac{\sqrt{\pi}}{2}(1+\erf(x))  
\end{equation*}
as well as $C=F_0(\infty)=\sqrt{\pi}$.
This implies that
\begin{align*}
\chi_0(\eta)&=\erf^{-1}\big(\frac{2}{\sqrt{\pi}}\eta-1 \big), \quad \eta\in(0,\sqrt{\pi}),\\
\U_0(\eta)&=\Big(\frac{\pi}{2}\Big)^{1/2} \erf\big(\frac{1}{\sqrt{2}}\erf^{-1}(\frac{2}{\sqrt{\pi}}\eta-1) \big),\quad \eta\in(0,\sqrt{\pi}).
\end{align*}
Considering the system of ordinary differential equations \eqref{HS:ode} with initial data $(\chi,\U)|_{t=0}=(\chi_0,\U_0)$, we find
\begin{align*}
\chi(t,\eta)&=\frac{t^2}{4}(\eta-\frac12 C)+\U_0(\eta)t+\chi_0(\eta) , \\
\U(t,\eta)&  =\frac{t}2(\eta-\frac12 C)+\U_0(\eta).
\end{align*}
See Figure \ref{HS_Smooth_erf}. Observe that it is not easy to transform this solution explicitly back to the original variable $u$.
\begin{figure}\centering
\includegraphics[width=10cm]{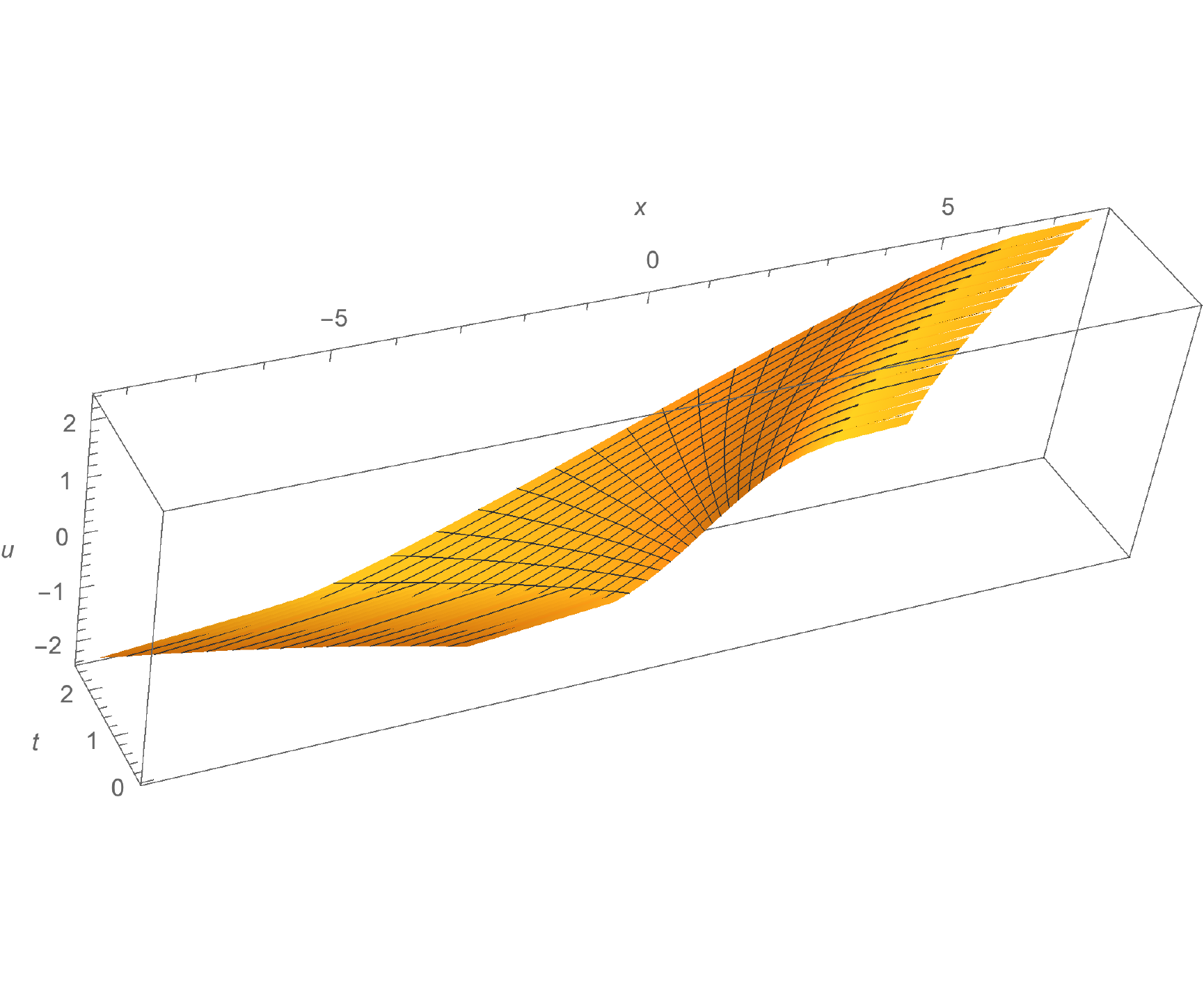}
\caption{The surface $\{(\chi,t,\U)\mid t\in[0,2.5], \, \eta\in(0,\sqrt{\pi})\}$ discussed in Example \ref{rem:smooth} (i).}
\label{HS_Smooth_erf}
\end{figure}

(ii) Let
\begin{equation*}
u_0(x)=\arcsinh(x), \quad \mu_0(x)=u_{0,x}^2(x)dx=\frac{dx}{1+x^2}.
\end{equation*}
Note that $u_0$ is not bounded, yet the same transformations apply.
We find that 
\begin{equation*}
F_0(x)=\mu_0((-\infty,x))= \arctan(x)+\frac\pi2 
\end{equation*}
as well as $C=F_0(\infty)=\pi$.
This implies that
$$
\chi_0(\eta)=\tan(\eta-\frac\pi2) \quad\text{ and } \quad \U_0(\eta)=\arcsinh(\tan(\eta-\frac\pi2)) \quad \text{ for } \eta\in(0,\pi).
$$
Here we find
\begin{align*}
\chi(t,\eta)&=\frac{t^2}{4}(\eta-\frac12 C)+\U_0(\eta)t+\chi_0(\eta) , \\
\U(t,\eta)&  =\frac{t}2(\eta-\frac12 C)+\U_0(\eta).
\end{align*}
See Figure \ref{HS_1solitonSmooth}. Again it is not easy to transform this solution explicitly back to the original variable $u$.
\begin{figure}\centering
\includegraphics[width=10cm]{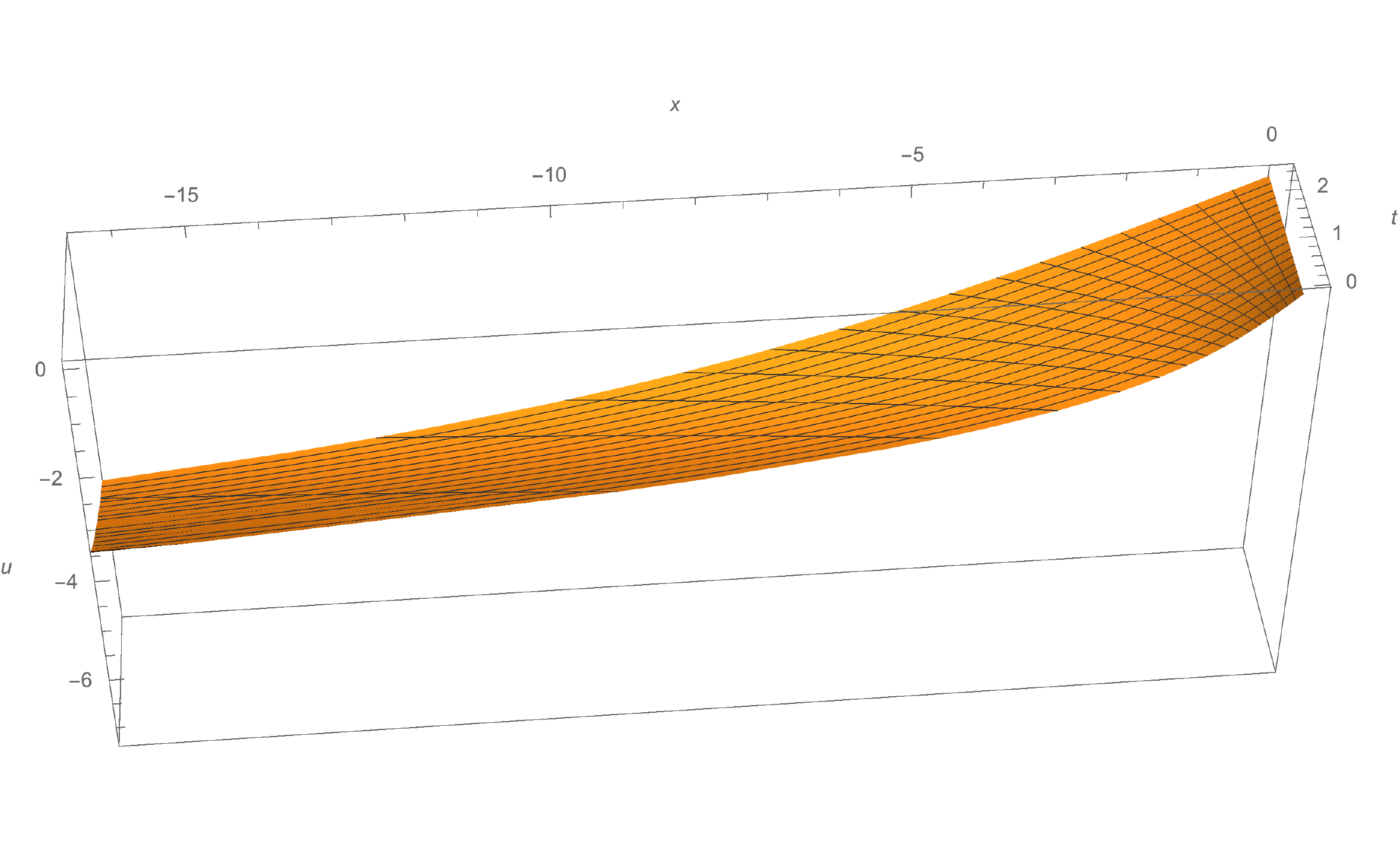}
\caption{The surface $\{(\chi,t,\U)\mid t\in[0,2.5], \, \eta\in(0,\pi)\}$ discussed in Example \ref{rem:smooth} (ii).}
\label{HS_1solitonSmooth}
\end{figure}
\end{example}

Let us next consider an example where the initial measure is a pure point measure.
\begin{example}\label{rem:soliton}
This simple singular example shows the interplay between measures $\mu$ and their pseudo inverses $\chi(x)$ better.\footnote{The solution in \eqref{eq:counter} comes from this example.} Consider the example $u_0=0$ and $\mu_0=\alpha\delta_0$, where 
$\delta_0$ is the Dirac delta function at the origin, and $\alpha\ge0$. Then $F_0\colon\Real\to [0,\alpha]$ reads
\begin{equation*}
F_0(x)=\begin{cases}
0, & \text{ if $x\leq 0$},\\
\alpha, & \text{ if  $x>0$}.
\end{cases}
\end{equation*}
The corresponding pseudo inverse $\chi_0\colon[0,\alpha]\to \Real$ is then given by
\begin{equation*}
\chi_0(\eta)=\begin{cases}
-\infty, & \text{ if $\eta=0$},\\
0, &  \text{ if $\eta\in (0,\alpha]$} .
\end{cases}
\end{equation*}
Thus\footnote{Note that in the smooth case both $\chi(F)$ and $F(\chi)$ are the identity function!}  
\begin{equation*}
\chi_0(F_0(x))=\begin{cases}
-\infty, & \text{ if $x\leq 0$},\\
0, & \text{ if  $x>0$},
\end{cases}\quad \text{ and } \quad F_0(\chi_0(\eta))= 0 \quad \text{ for all } \eta\in[0,\alpha].
\end{equation*}
In general one observes that jumps in $F_0(x)$ are mapped to intervals where $\chi_0(\eta)$ is constant and vice versa. This means in particular that intervals where $F_0(x)$ is constant shrink to single points. Moreover, if $F_0(x)$ is constant on some interval, then $u_0(x)$ is also constant on the same interval.

Next we compute the time evolution of both $\chi(t,\eta)$ and $\U(t,\eta)=u(t,\chi(t,\eta))$. Following the approach in \cite{BHR}, we obtain that the corresponding solution in Eulerian coordinates reads for $t$ positive
\begin{subequations}\label{HS:odeEX1_euler}
\begin{align}
u(t,x)&=\begin{cases} 
-\frac\alpha4 t, &  \text{ if $x\leq -\frac\alpha8 t^2$},\\
\frac{2x}{t},&  \text{ if $-\frac\alpha8 t^2\leq x\leq \frac\alpha8 t^2$},\\
\frac\alpha4 t, &  \text{ if $x\geq \frac\alpha8 t^2$},
\end{cases}\\
\mu(t,x)&= u_{x}^2(t,x) dx=\frac{4}{t^2}\indicator_{[-\alpha t^2/8,\alpha t^2/8]}(x)dx, \\
F(t,x)& =\begin{cases}
0, & \text{ if $x\leq -\frac\alpha8 t^2$},\\
\frac{4x}{t^2}+\frac\alpha2, &  \text{ if $-\frac\alpha8 t^2\leq x\leq \frac\alpha8 t^2$},\\
\alpha, &  \text{ if  $x\ge\frac\alpha8 t^2$}.
\end{cases}
\end{align}
\end{subequations}
Calculating the pseudo inverse $\chi(t,\eta)$ and $\U(t,\eta)=u(t,\chi(t,\eta))$ for each $t$ then yields
\begin{subequations}\label{HS:odeEX1}
\begin{align}
\chi(t,\eta)& = \frac{t^2}{4}\left(\eta-\frac\alpha2\right), \quad \eta\in(0,\alpha], \\
\U(t,\eta)& = \frac{t}{2}\left(\eta-\frac\alpha2\right), \quad \eta\in[0,\alpha],
\end{align}
\end{subequations}
and, in particular, that 
\begin{equation*}
\U_t(t,\eta) = \frac12 \left(\eta-\frac\alpha2\right), \quad \eta\in[0,\alpha].
\end{equation*}
Thus we still obtain the same ordinary differential equation \eqref{HS:ode} as in the smooth case! 
In addition, note that $\chi_t(0,\eta)=0$ for all $\eta\in(0,\alpha]$, and hence the important information is encoded in $\U_t(t,\eta)$.
\begin{figure}\centering
\includegraphics[width=10cm]{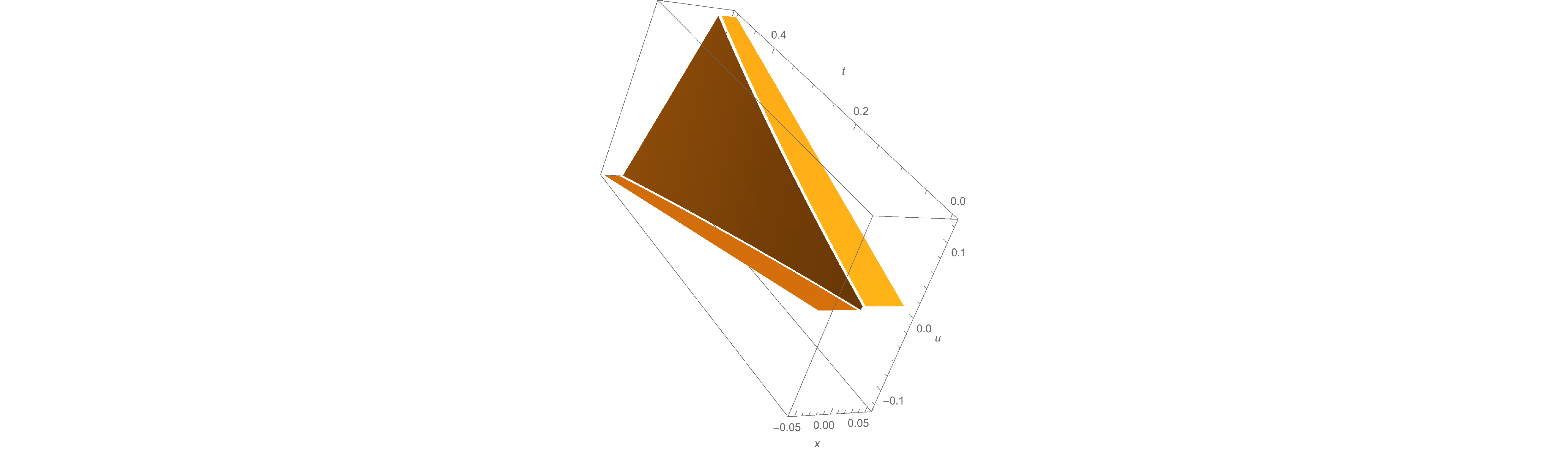}
\caption{The surface $\{(x,t,u)\mid t\in[0,.5], \, x\in[-.05,.05]\}$ discussed in Example \ref{rem:soliton}.}
\label{HS_solitondelta}
\end{figure}

We can of course also solve  $\chi_t=\U$ and $\U_t=\frac{\eta}{2}-\frac{\alpha}{4}$ directly with initial data $\chi_0=\U_0=0$, which again yields 
\eqref{HS:odeEX1}.  To return to the Eulerian variables $u$ and $\mu$ we have in the smooth region that
\begin{equation*}
u(t,x)=\U(t,\eta), \quad x=\chi(t,\eta), \quad \eta\in (0,\alpha], 
\end{equation*}
and we need to extend $\U$ and $\chi$ to all of  $\Real$ by continuity: 
\begin{align*}
\chi(t,\eta)&=\begin{cases} 
-\frac{\alpha }8t^2+\eta, &  \text{ if } \eta\leq 0,\\
\frac{t^2}{4}\left(\eta-\frac\alpha2\right),& \text{ if } \eta\in[0,\alpha], \\
\frac{\alpha}8t^2+\eta-\alpha, &  \text{ if } \eta\geq \alpha,
\end{cases} \quad  t\ge0, \\
\U(t,\eta)&=\begin{cases} 
-\frac{\alpha}4t=\U(t,0+), &  \text{ if } \eta\leq 0,\\
\frac{t}{2}\left(\eta-\frac\alpha2\right),& \text{ if } \eta\in[0,\alpha], \\
\frac{\alpha}4t=\U(t,\alpha-), &  \text{ if } \eta\geq \alpha,
\end{cases} \quad  t\geq0.
\end{align*}
Returning to the Eulerian variables we recover \eqref{HS:odeEX1_euler}. We can also depict the full solution in the $(x,t)$ plane in the new variables:  The full solution reads
\begin{equation*}
\{(\chi(t,\eta),t, \U(t,\eta))\in\Real^3 \mid t\in (0,\infty), \quad \eta\in \Real\}.
\end{equation*}
See Figure \ref{HS_solitondelta}.
\end{example}

The next example shows the difficulties that one has to face in the general case where the solution encounters a break down in the sense of steep gradients. 
 \begin{example}\label{rem:smoothA} 
Let
\begin{equation*}
u_0(x)=-x\indicator_{[0,1]}(x)- \indicator_{[1,\infty)}(x), \quad \mu_0(x)=u_{0,x}^2(x)dx=\indicator_{[0,1]}(x)dx.
\end{equation*}
Next we find
\begin{align*}
F_0(x)&= x\indicator_{[0,1]}(x)+  \indicator_{[1,\infty)}(x), \\
\chi_0(\eta)&=\begin{cases}
-\infty, & \text{ if $\eta= 0$},\\
\eta, & \text{ if  $\eta\in(0,1]$},
\end{cases} \\
\U_0(\eta)&  =-\eta, \quad \eta\in[0,1].
\end{align*}
\textit{Assuming} \eqref{HS:ode} holds also in this case,  we find
\begin{align*}
\chi(t,\eta)&=\frac{t^2}{4}(\eta-\frac12)-t\eta+\eta, \quad \eta\in(0,1],\\
\U(t,\eta)&  =\frac{t}{2}(\eta-\frac12)-\eta, \quad \eta\in[0,1].
\end{align*}
We extend the functions by continuity
\begin{align*}
\chi(t,\eta)&=\begin{cases} -\frac{t^2}8+\eta, &  \text{ if $\eta\leq 0$},\\
\frac{t^2}{4}(\eta-\frac12)-t\eta+\eta,& \text{ if $\eta\in[0,1]$}, \\
\frac{t^2}8-t+\eta, &  \text{ if $\eta\geq 1$},
\end{cases} \\
\U(t,\eta)&=\begin{cases} -\frac{t}4, &  \text{ if $\eta\leq 0$},\\
\frac{t}{2}(\eta-\frac12)-\eta,& \text{ if $\eta\in[0,1]$}, \\
\frac{t}4-1, &  \text{ if $\eta\geq 1$},
\end{cases}.
\end{align*}
which gives a well-defined global solution given by $\{(\chi(t,\eta),t,\U(t,\eta))\mid t\ge0, \, \eta\in\Real\}$.  However, as we return to Eulerian variables, the time-development is more dramatic. Solving the equation $x=\chi(t,\eta)$ for $\eta\in[0,1]$ yields
\begin{equation*}
\eta=\frac{4x+t^2/2}{(t-2)^2}\in[0,1],
\end{equation*}
which leads to
the solution 
\begin{equation*}
u(t,x)= \U\left(t,\frac{4x+t^2/2}{(t-2)^2}\right)= \frac{2x+t/2}{t-2}
\end{equation*}
whenever
\begin{equation*}
-\frac{t^2}{8}<x<\frac14\left((t-2)^2-\frac{t^2}{2}\right).
\end{equation*}
For $t\to2-$, we have that $u_x\to -\infty$ at $x=-1/2$.  
The solution on the full line reads  
\begin{equation*}
u(t,x)=\begin{cases} 
-\frac14 t, &  \text{ if $x\leq -\frac18 t^2$},\\
\frac{2x+t/2}{t-2},&  \text{ if $-\frac18 t^2\leq x\leq \frac14( (t-2)^2-\frac{t^2}2)$},\\
\frac14 t-1, &  \text{ if $x \ge \frac14( (t-2)^2-\frac{t^2}2)$}.
\end{cases}
\end{equation*}
The solution is illustrated in Figure \ref{HS_1break}.

\begin{figure}\centering
\includegraphics[width=10cm]{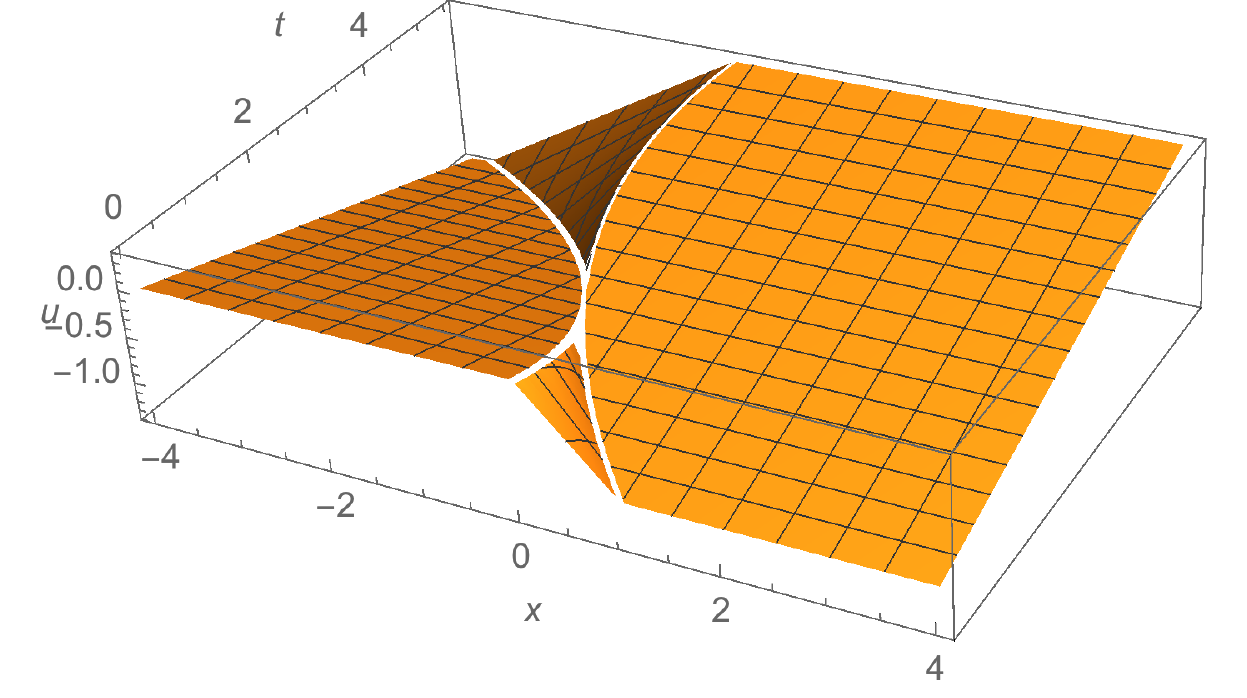}\\
\includegraphics[width=2.5cm]{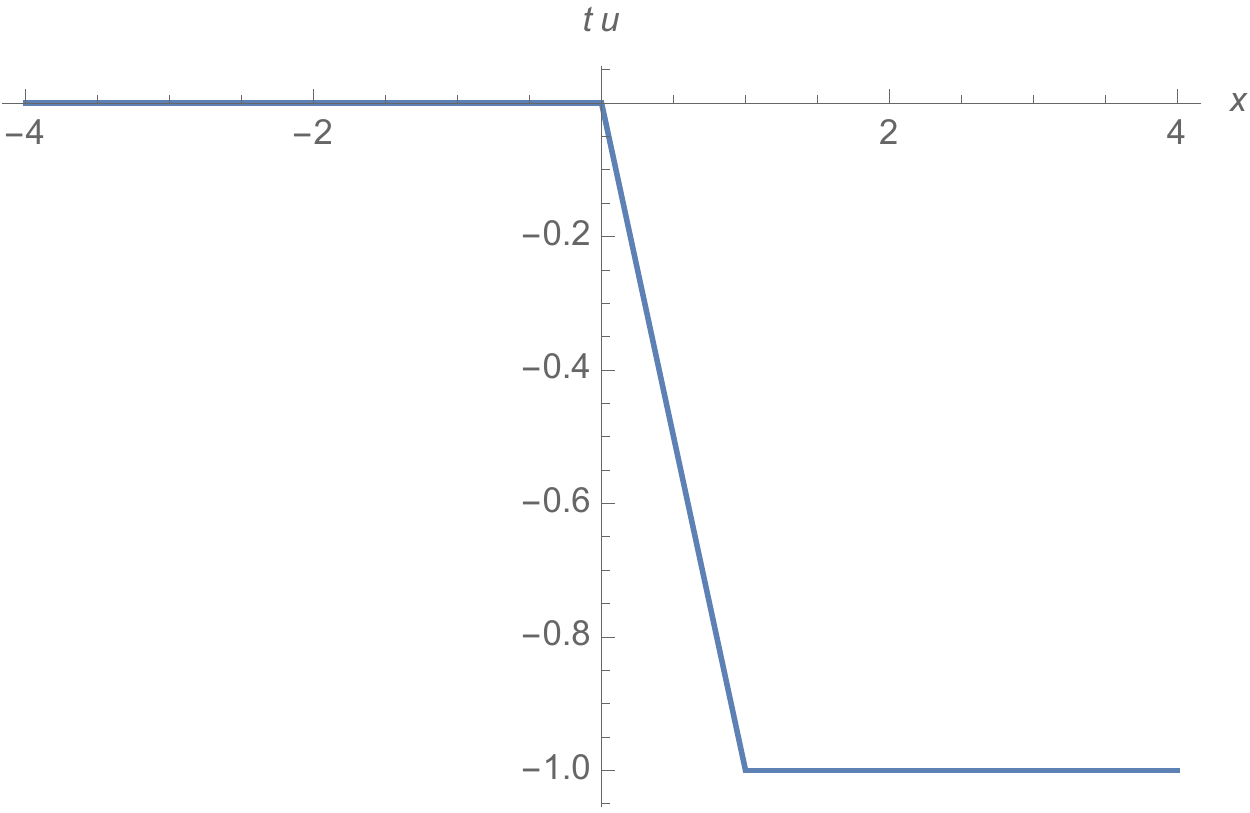}
\includegraphics[width=2.5cm]{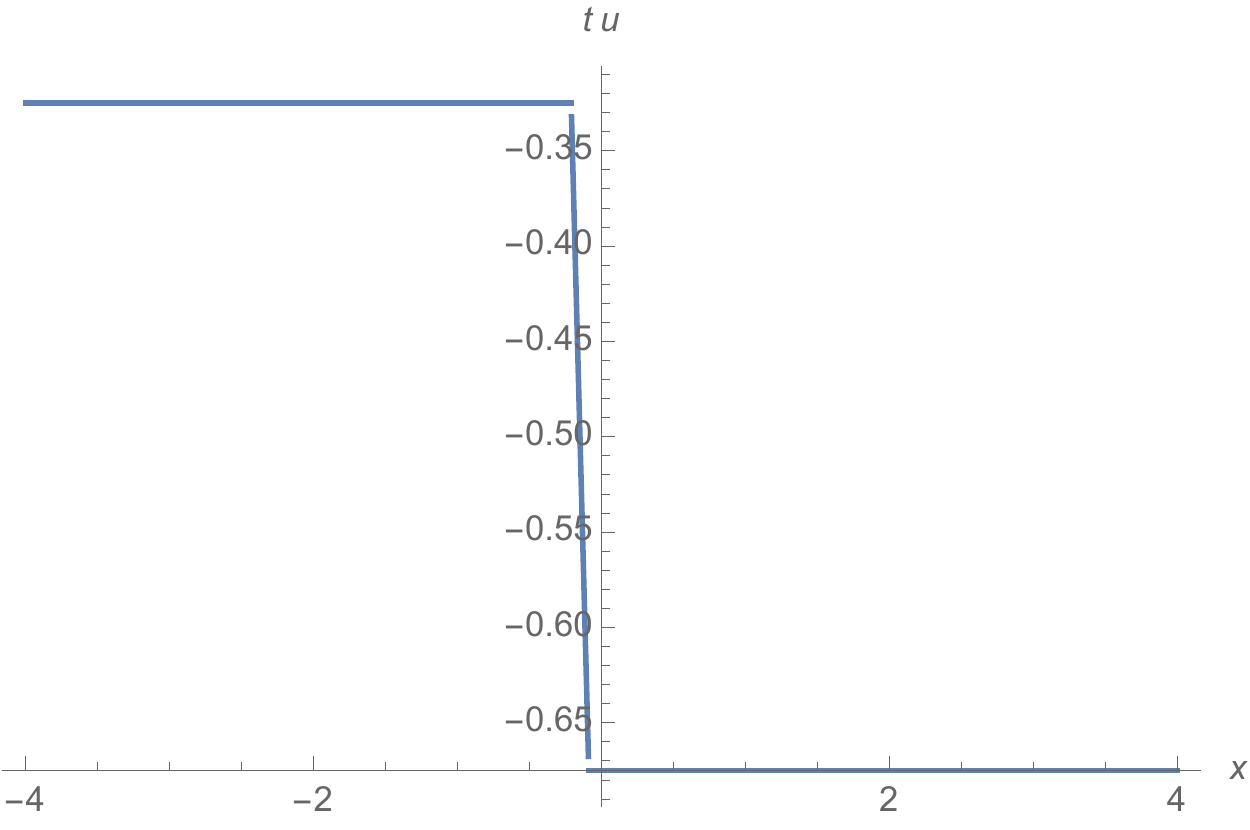}
\includegraphics[width=2.5cm]{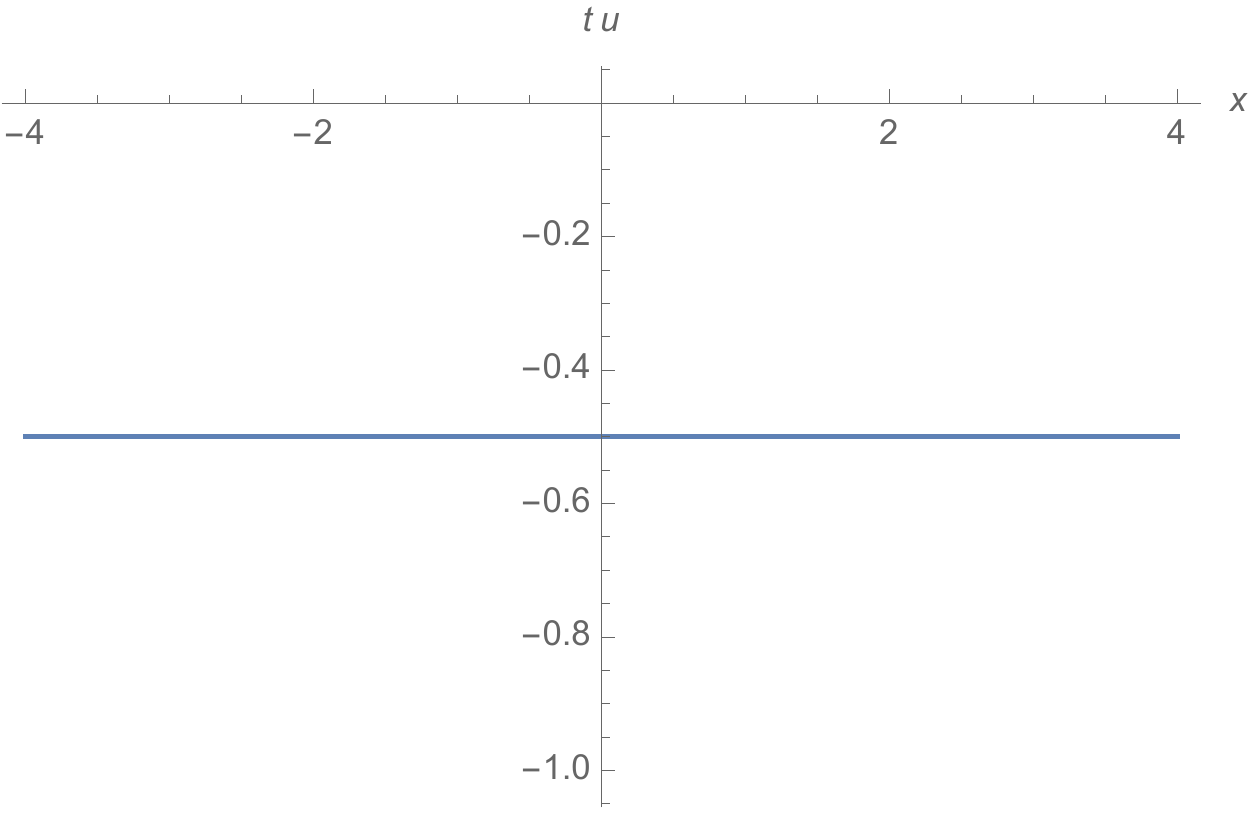}
\includegraphics[width=2.5cm]{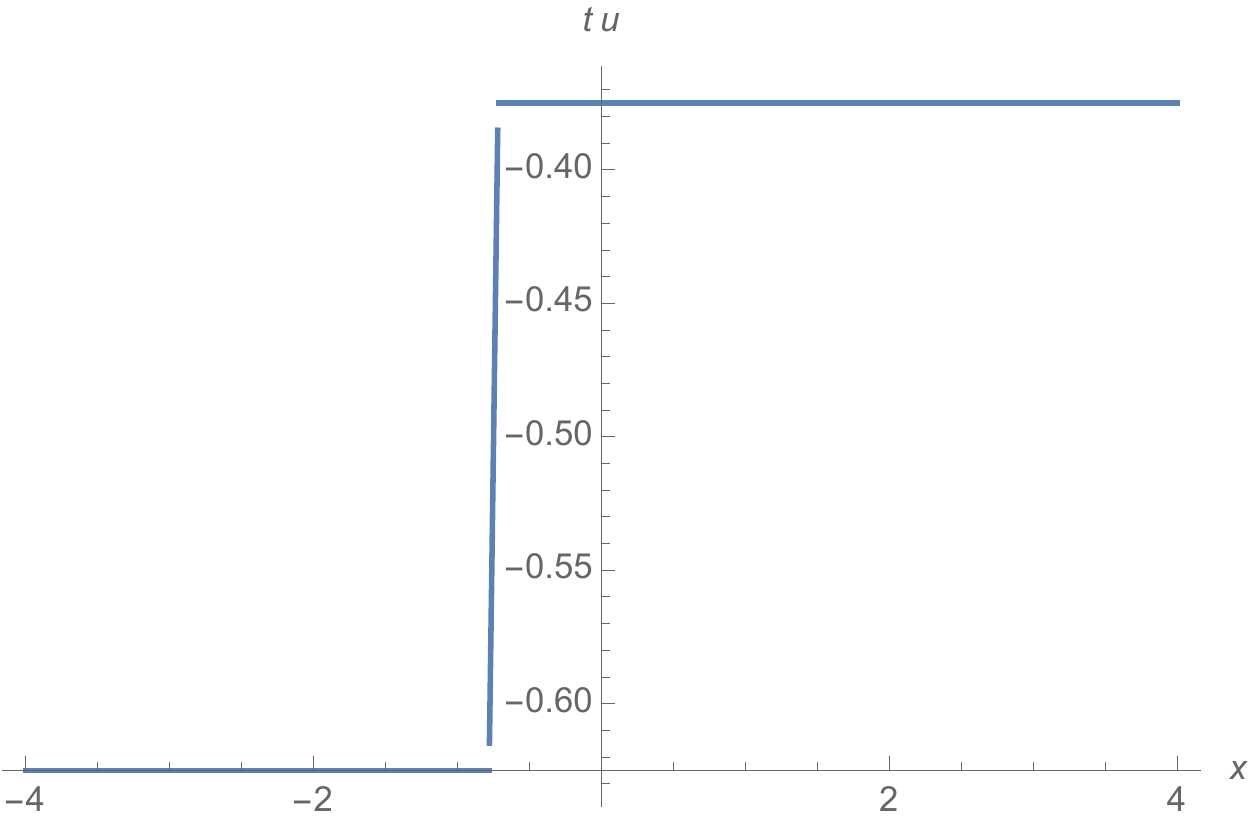}
\caption{The solution discussed in Example \ref{rem:smoothA}.}
\label{HS_1break}
\end{figure}
\end{example}

The above examples already hint that the interplay between Eulerian and Lagrangian coordinates is going to play a major role in our further considerations. 
We assume a smooth solution of
\begin{subequations} 
\begin{align}
  u_t+uu_x&=\frac14\left(\int_{-\infty}^xu_x^2(y)\,dy-\int_{x}^\infty u_x^2(y)\,dy\right), \label{eq:hs}\\
  (u_x^2)_t+(uu_x^2)_x&=0. \label{eq:transport}
\end{align}
\end{subequations}

Next, we rewrite the equation in Lagrangian coordinates. 
Introduce the characteristics
\begin{equation*}
  y_t(t,\xi)=u(t,y(t,\xi)).
\end{equation*}
The  Lagrangian velocity  $U$ reads
\begin{equation*}
  U(t,\xi)=u(t,y(t,\xi)).
\end{equation*}
Furthermore, we define the Lagrangian cumulative energy by
\begin{equation*}
  H(t,\xi)=\int_{-\infty}^{y(t,\xi)}u_x^2(t,x)\,dx.
\end{equation*}
From \eqref{eq:hs}, we get that
\begin{equation*}
  U_t=u_t\circ y+y_tu_x\circ y=\frac14\left(\int_{-\infty}^yu_x^2\,dx-\int_{y}^\infty u_x^2\,dx\right)=\frac12H-\frac14 C
\end{equation*}
where $C=H(t,\infty)$ is time independent,
and
\begin{align*}
  H_t&=\int_{-\infty}^{y(t,\xi)}(u_x^2(t,x))_t\,dx+y_tu_x^2(t,y)=\int_{-\infty}^{y(t,\xi)}((u_x^2)_t+(uu_x^2)_x)(t,x)\,dx=0
\end{align*}
by \eqref{eq:transport}. In this formal
computation, we require that $u$ and $u_x$ are
smooth and decay rapidly at infinity. Hence, the
HS equation formally is equivalent to
the following system of ordinary differential
equations:
\begin{subequations}
  \label{eq:sys}
  \begin{align}
    \label{eq:sys1}
    y_t&=U,\\
    \label{eq:sys2}
    U_t&=\frac12H-\frac14C,\\
    \label{eq:sys3}
    H_t&=0.
  \end{align}
\end{subequations}
Global existence of
solutions to \eqref{eq:sys} follows from the
linear nature of the system. There is no exchange
of energy across the characteristics and the
system \eqref{eq:sys} can be solved explicitly. This is in contrast to the Camassa--Holm equation where energy is exchanged across characteristics. We
have
\begin{subequations}
  \begin{align*}
    y(t,\xi)&=\left( \frac14H(0,\xi)-\frac18C\right) t^2+U(0,\xi)t+y(0,\xi),\\
    U(t,\xi)&=\left( \frac12H(0,\xi)-\frac14C\right) t+U(0,\xi),\\
    H(t,\xi)&=H(0,\xi).
  \end{align*}
\end{subequations}

We next focus on the general case without assuming regularity of the solution. It turns out that in addition to the variable $u$ we will need a measure $\mu$ that in smooth regions coincides with the energy density $u_x^2 dx$. At wave breaking, the energy at the point where the wave breaking takes place, is transformed into a point measure. It is this dynamics that is encoded in the measure $\mu$ that allows us to treat general initial data. An important complication stems from the fact that the original solution in two variables $(u,\mu)$ is transformed into Lagrangian coordinates with three variables $(y,U,H)$. This is a well-known consequence of the fact that one can parametrize a particle path in several different ways, corresponding to the same motion. This poses technical complications when we want to measure the distance between two distinct solutions in Lagrangian coordinates that correspond to the same Eulerian solution, and we denote this as \textit{relabeling} of the solution. 

We will employ the notation and the results from  \cite{BHR} and \cite{anders}.  Define the Banach spaces
\begin{align*}
  E_1&=\{f\in L^\infty(\Real) \mid f'\in L^2(\Real), \, \lim_{\xi\to-\infty}f(\xi)=0\}, \\
   E_2&=\{f\in L^\infty(\Real) \mid f'\in L^2(\Real)\},
\end{align*}
with  norms
\begin{equation*}
  \norm{f}_{E_j}=\norm{f}_{L^\infty}+\norm{f'}_{L^2}, \quad f\in E_j,\, j=1,2.
\end{equation*}
Let 
\begin{equation*}
  B=E_2\times E_2\times E_1,
\end{equation*}
with norm
\begin{equation*}
  \norm{(f_1,f_2,f_3)}_{B}=\norm{f_1}_{E_2}+\norm{f_2}_{E_2}+ \norm{f_3}_{E_1}, \quad (f_1,f_2,f_3)\in B.
\end{equation*}

We are given some initial data 
$(u_0,\mu_0)\in \D$, where the set $\D$ is defined as follows.
\begin{definition}
  The set $\D$ consists of all pairs $(u,\mu)$
  such that\\[1mm]
  (i) $u\in E_2$;\\[1mm]
(ii) $\mu$ is a nonnegative and finite Radon measure such that
$\muac=u_x^2dx$ where $\muac$ denotes the absolute continuous
  part of $\mu$ with respect to the Lebesgue
  measure.
\end{definition}

The Lagrangian variables are given by $(\zeta,U,H)$ (with $\zeta=y-\id$), and the appropriate space is defined as follows.
\begin{definition}
  \label{def:F}
  The set $\F$ consists of the elements
  $(\zeta,U,H)\in B=E_2\times E_2\times E_1$ such
  that  \\[1mm]
  (i) $(\zeta,U,H)\in(W^{1,\infty}(\Real))^3$, where $\zeta(\xi)=y(\xi)-\xi$; \\[1mm]
  (ii) $y_\xi\geq0$, $H_\xi\geq0$ and $y_\xi+H_\xi\geq c$, almost everywhere, where  $c$ is a strictly positive constant;\\[1mm]
  (iii) $y_\xi H_\xi=U_\xi^2$ almost everywhere.
\end{definition}
The key subspace $\F_0\subset\F$ is defined by
\begin{equation*}
  \F_0=\{X=(y,U,H)\in\F\mid y+H=\id\}.
\end{equation*}

We need to clarify the relation between the Eulerian variables $(u,\mu)$ and the Lagrangian variables $(\zeta,U,H)$.
The transformation 
\begin{equation*}
L\colon \D\to \F_0, \quad X=L(u,\mu)
\end{equation*}
is defined as follows.  
\begin{definition}
  \label{def:Ldef}
  For any $(u,\mu)$ in $\D$, let
  \begin{subequations}
    \label{eq:Ldef}
    \begin{align}
      \label{eq:Ldef1}
      y(\xi)&=\sup\left\{y\mid \mu((-\infty,y))+y<\xi\right\},\\
      \label{eq:Ldef2}
      H(\xi)&=\xi-y(\xi),\\
      \label{eq:Ldef3}
      U(\xi)&=u\circ{y(\xi)}.
    \end{align}
  \end{subequations}
  Then $X=(\zeta,U,H)\in\F_0$ and we denote by
  $L\colon \D\to\F_0$ the map which to any
  $(u,\mu)\in\D$ associates $(\zeta,U,H)\in\F_0$ as
  given by \eqref{eq:Ldef}.
\end{definition}
From the Lagrangian variables we can return to Eulerian variables using the following transformation.
\begin{definition}
  \label{def:Mdef}
  Given any element $X$ in $\F$. Then, the pair
  $(u,\mu)$ defined as follows
  \begin{subequations}
    \label{eq:umudef}
    \begin{align}
      \label{eq:umudef1}
      &u(x)=U(\xi)\text{ for any }\xi\text{ such that  }  x=y(\xi),\\
      \label{eq:umudef2}
      &\mu=y_\#(H_\xi\,d\xi)
    \end{align}
  \end{subequations}
  belongs to $\D$. Here, the push-forward  of a measure
    $\nu$ by a measurable function $f$ is the measure $f_\#\nu$ defined by
    $f_\#\nu(B)=\nu(f\inv(B))$ for all Borel sets $B$. We denote by
  $M\colon\F\rightarrow\D$ the map which to any
  $X$ in $\F$ associates $(u,\mu)$ as given by
  \eqref{eq:umudef}.
\end{definition}
The key properties of these transformations are (cf. \cite[Prop. 2.11]{BHR})
\begin{equation}\label{eq:BHR3}
L\circ M|_{\F_0}=\id_{\F_0}, \quad M\circ L=\id_\D. 
\end{equation}

The formalism up to this point has been stationary, transforming back and forth between Eulerian and  Lagrangian variables. Next we can take into consideration the time-evolution of the solution of the HS equation.

The evolution of the HS equation in Lagrangian variables is determined by the system  (cf. \eqref{eq:sys})
\begin{equation}\label{eq:BHR4}
S_t\colon \F\to \F, \quad X(t)=S_t(X_0), \quad X_t=S(X), \quad X|_{t=0}=X_0.
\end{equation}
of ordinary differential equations. Here
\begin{equation*}
S(X)= \begin{pmatrix}
U \\ \frac12H-\frac14C \\ 0
\end{pmatrix}.
\end{equation*}
Next, we address the question about relabeling. We need to identify Lagrangian solutions that correspond to one and the same solution in Eulerian coordinates.  Let 
 $\Gr$ be the subgroup of the
group of homeomorphisms on $\Real$
such that
\begin{subequations}
\begin{align*}
  f-\id\text{ and }f^{-1}-\id&\text{ both belong to }W^{1,\infty}(\Real),\\
  f_\xi-1 &\text{ belongs to } L^2(\Real).
\end{align*}
\end{subequations}

By default the HS equation is invariant under relabeling, which is given by  equivalence classes 
\begin{align*}
 [X]&=\{\tilde X\in \F\mid \text{there exists $g\in G$ such that  $X=\tilde X\circ g$}\},\\
 \F/G&=\{[X]\mid X\in \F\}.
\end{align*}
The key subspace of $\F$ is denoted $\F_0$ and is defined by
\begin{equation*}
  \F_0=\{X=(y,U,H)\in\F\mid y+H=\id\}.
\end{equation*}
The map into the critical space $\F_0$ is taken care of by (cf. \cite[Def. 2.9]{BHR})
\begin{equation*}
\Pi\colon \F\to\F_0, \quad \Pi(X)=X\circ(y+H)^{-1},
\end{equation*}
with the property that $\Pi(\F)=\F_0$.  We note that the map $X\mapsto [X]$ from $\F_0$ to $\F/G$ is a bijection.   
Then we have that (cf. \cite[Prop. 2.12]{BHR})
\begin{equation*}
\Pi\circ S_t\circ \Pi= \Pi\circ S_t, 
\end{equation*}
and hence we can define the semigroup
\begin{equation}\label{eq:BHR7}
\tilde S_t=\Pi\circ S_t\colon \F_0\to\F_0.
\end{equation}

We can now provide the solution of the HS equation. 
Consider initial data $(u_0,\mu_0)\in \D$, and define $\bar X_0=(\bar y_0, \bar U_0,\bar H_0)=L(u_0,\mu_0)\in\F_0$ given by  
\begin{subequations} 
\begin{align*}
\bar y_0(\xi)& =\sup\{x\mid x+F(0,x)<\xi\},\\
\bar U_0(\xi)& = u_0(\bar y_0(\xi)), \\
\bar H_0(\xi)&=\xi-\bar y_0(\xi),
\end{align*}
\end{subequations}
with $F(0,x)=\mu_0((-\infty,x))$.   
Next we want to determine the solution $(u(t),\mu(t))\in \D$ (we suppress the dependence in the notation on the spatial variable $x$ when convenient) for arbitrary time $t$. 

Define 
\begin{equation}\label{eq:BHR8}
\bar X(t)=S_t \bar X_0\in \F, \quad X(t)=\tilde S_t \bar X_0\in \F_0.
\end{equation}
The advantage of $\bar X(t)$ is that it obeys the differential equation \eqref{eq:BHR4}, while $X(t)$ keeps the relation $y+H=\id$ for all times.
From \eqref{eq:BHR7} we have that 
\begin{equation*}
 X(t)=\Pi(\bar X(t)).
\end{equation*}
We know that $\bar X(t,\xi)=(\bar y(t,\xi), \bar U(t,\xi), \bar H(t,\xi))\in \F$  is the solution of 
\begin{subequations}\label{eq:tildenull}
\begin{align}
\bar y_t(t,\xi)&=\bar U(t,\xi),\\
\bar U_t(t,\xi)& = \frac12 \bar H(\xi)-\frac14 C, \\
\bar H_t(t,\xi)& =0,
\end{align}
\end{subequations}
where $C=\mu_0(\Real)$ and $\bar X(0)=\bar X_0$.  
Straightforward integration yields
\begin{subequations}
\begin{align*}
\bar y(t,\xi)&=\frac14(\bar H_0(\xi)-\frac12 C)t^2+\bar U_0(\xi)t+\bar y_0(\xi),\\
\bar U(t,\xi)& = \frac12 (\bar H_0(\xi)-\frac12 C)t+\bar U_0(\xi),  \\
\bar H(t,\xi)& =\bar H_0(\xi)=\xi-\bar y_0(\xi).
\end{align*}
\end{subequations}
The solution $(u(t),\mu(t))=M(\bar X(t))$ in Eulerian variables reads
\begin{subequations}
\begin{align*}
u(t,x)&= \bar U(t,\xi), \quad \bar y(t,\xi)=x, \\
\mu(t,x)&= \bar y_\#(\bar H_\xi(t,\xi)d\xi), 
\end{align*}
\end{subequations} 
with $F(t,x)=\mu(t, (-\infty,x))=\int_{\bar y(t,\xi)<x} (1- \bar y_{0,\xi}(\xi))d\xi$.

However, for $X(t,\xi)=(y(t,\xi), U(t,\xi), H(t,\xi))\in\F_0$, which satisfies
\begin{equation*}
 X(t)=\bar X(t)\circ (\bar y+\bar H)^{-1}\in\F_0,
\end{equation*}
we see, using \eqref{eq:BHR3}, that
\begin{align*}
 y(t,\xi)&=\sup\{x\mid x+F(t,x)< \xi\} \notag\\
&=\sup\{x\mid x+F(t,x)< y(t,\xi)+H(t,\xi)\}, 
\end{align*}
where we in the second equality use that $X(t)\in\F_0$. 
Note that we still have $(u(t),\mu(t))=M(\bar X(t))= M(X(t))$, and thus 
\begin{subequations}
\begin{align*}
u(t,x)&= U(t,\xi), \quad y(t,\xi)=x, \\
\mu(t,x)&= y_\#(H_\xi(t,\xi)d\xi), 
\end{align*}
\end{subequations} 
with $F(t,x)=\mu(t, (-\infty,x))=\int_{y(t,\xi)<x} H_\xi(t,\xi)d\xi$.
Since $\bar X(t)=X(t)\circ (\bar y+\bar H)$, we find that
\begin{equation}
\bar y(t,\xi)=y(t,\bar y(t,\xi)+\bar H(t,\xi))
=\sup\{x\mid x+F(t,x)< \bar y(t,\xi)+\bar H(t,\xi)\}.  \label{eq:BHR12}
\end{equation}
This is the only place in this construction where we use the quantity $X(t)$.

\medskip
Define now
\begin{equation*}
\chi(t,\eta)=\sup\{x\mid\mu(t,(-\infty,x))<\eta\}=\sup\{x\mid F(t,x)<\eta\}.
\end{equation*}
We claim that
\begin{equation*}
\chi(t,\eta) =\bar y(t, l(t,\eta)),
\end{equation*}
where we have introduced $l(t,\dott)\colon[0,C]\to\Real$ by 
\begin{equation}\label{def:l}
l(t,\eta)=\sup\{\xi\mid \bar H(t,\xi)<\eta\}.
\end{equation}
Note that since $ \bar H_t=0$, we have that
\begin{equation*}
l(t,\eta)=l(0,\eta) \quad \text{ and }\quad  l_t(t,\eta)=0.
\end{equation*}
Recall that for each time $t$ we have (cf.~\eqref{eq:BHR12})  
\begin{equation*}
\bar y(t, \xi)=\sup\{x\mid x+F(t,x)< \bar y(t,\xi)+\bar H(t,\xi)\},
\end{equation*}
which implies that  
\begin{equation*}
\bar y(t,\xi)+F(t,\bar y(t,\xi))\leq \bar y(t,\xi)+\bar H(t,\xi)\leq \bar y(t,\xi)+F(t,\bar y(t,\xi)+).
\end{equation*}
Subtracting $\bar y(t,\xi)$ in the above inequality, we end up with 
\begin{equation*}
F(t, \bar y(t,\xi))\leq \bar H(t,\xi)\leq F(t,\bar y(t,\xi)+) \quad \text{ for all } \xi\in\Real.
\end{equation*}
Comparing the last equation and \eqref{def:l}, we have 
\begin{equation*}
F(t, \bar y(t, l(t,\eta)))\leq \bar H(t, l(t,\eta))=\eta\leq F(t, \bar y(t,l(t,\eta))+).  
\end{equation*}
Since $\bar y(t,\dott)$ is surjective and non-decreasing, we end up with 
\begin{equation*}
\chi(t,\eta)=\sup\{x\mid F(t,x)<\eta\}= \bar y(t,l(t,\eta)).
\end{equation*}

Introduce the new function
\begin{equation*}
\U(t,\eta) = \bar U(t, l(t,\eta)).
\end{equation*}
We are now ready to derive the system of ordinary differential equations for $\chi(t,\eta)$ and $\U(t,\eta)$. Therefore recall that 
\begin{equation*}
\bar H(t,l(t,\eta))=\bar H(0,l(t,\eta))=\bar H(0,l(0,\eta))=\eta \quad \text{ for all } \eta\in [0,C],
\end{equation*}
since $\bar H(0,\xi)$ is continuous. Direct calculations yield
\begin{align*}
\chi_t(t,\eta)&=\frac{d}{dt}\bar y(t, l(0,\eta))=\bar y_t(t,l(0,\eta))=\bar U(t, l(0,\eta))=\bar U(t, l(t,\eta))\\
&=\U(t,\eta),\\
\U_t(t,\eta)&= \frac{d}{dt}\bar U(t,l(0,\eta))= \bar U_t(t, l(0,\eta))= \frac12 \bar H(t,l(0,\eta))-\frac14 C\\
&= \frac 12 \eta-\frac14 C.
\end{align*}

Thus we established rigorously the linear system 
\begin{subequations}\label{sys:ode:metric}
\begin{align}
\chi_t(t,\eta)&= \U(t,\eta),\\
\U_t(t,\eta)&= \frac12 \eta-\frac14 C
\end{align}
\end{subequations}
of ordinary differential equations, with solution
\begin{subequations}
\begin{align*}
\chi(t,\eta)&=\left(\frac14\eta-\frac18 C\right)t^2+\U(0,\eta)t+\chi(0,\eta),\\
\U(t,\eta)&=\left(\frac12 \eta-\frac14 C\right)t+\U(0,\eta).
\end{align*}
\end{subequations}

\begin{example}
Recall Example \ref{rem:soliton}.  Let us compute the corresponding quantities in the case with $u_0=0$ and $\mu_0=\delta_0$.
Here we find  that
\begin{align*}
\bar y_0(\xi)&=\begin{cases}
\xi, & \text{for $\xi\le 0$}, \\
0, & \text{for $\xi\in(0,1)$}, \\
\xi-1, & \text{for $\xi\ge1$}, 
\end{cases} \\[2mm]
\bar U_0(\xi)&=0, \\[2mm]
\bar H_0(\xi)&=\begin{cases}
0, & \text{for $\xi\le 0$}, \\
\xi, & \text{for $\xi\in(0,1)$}, \\
1, & \text{for $\xi\ge1$}.
\end{cases} 
\end{align*}
The solution of $\bar X=S_t\bar X_0$  (cf.~\eqref{eq:BHR8}) reads
\begin{align*}
\bar y(t,\xi)&=\begin{cases}
\xi-\frac{t^2}{8}, & \text{for $\xi\le 0$}, \\
\frac{t^2}{4}\xi-\frac{t^2}{8}, & \text{for $\xi\in(0,1)$}, \\
\frac{t^2}{8}+\xi-1, & \text{for $\xi\ge1$}, 
\end{cases} \\[2mm]
\bar U(t,\xi)&=\begin{cases}
-\frac{t}{4}, & \text{for $\xi\le 0$}, \\
\frac{t}{2}\xi-\frac{t}{4}, & \text{for $\xi\in(0,1)$}, \\
\frac{t}{4}, & \text{for $\xi\ge1$}, 
\end{cases} \\[2mm]
\bar H(t,\xi)&=\begin{cases}
0, & \text{for $\xi\le 0$}, \\
\xi, & \text{for $\xi\in(0,1)$}, \\
1, & \text{for $\xi\ge1$}.
\end{cases} 
\end{align*}
If we compute the corresponding quantities $X(t)=\bar X(t)\circ(\bar y+\bar H)^{-1}$, we find
\begin{align*}
y(t,\xi)&=\begin{cases}
\xi, & \text{for $\xi\le -\frac{t^2}{8}$}, \\
\frac{t^2}{t^2+4}\xi-\frac{t^2}{2t^2+8}, & \text{for $\xi\in(-\frac{t^2}{8},\frac{t^2}{8}+1)$}, \\
\xi-1, & \text{for $\xi\ge\frac{t^2}{8}+1$}, 
\end{cases} \\[2mm]
U(t,\xi)&=\begin{cases}
-\frac{t}{4}, & \text{for $\xi\le -\frac{t^2}{8}$}, \\
\frac{2t}{t^2+4}(\xi+\frac{t^2}{8})-\frac{t}4, & \text{for $\xi\in(-\frac{t^2}{8},\frac{t^2}{8}+1)$}, \\
\frac{t}{4}, & \text{for $\xi\ge\frac{t^2}{8}+1$}, 
\end{cases} \\[2mm]
H(t,\xi)&=\begin{cases}
0, & \text{for $\xi\le -\frac{t^2}{8}$}, \\
\frac{4}{t^2+4}(\xi+\frac{t^2}{8}), & \text{for $\xi\in(-\frac{t^2}{8},\frac{t^2}{8}+1)$}, \\
1, & \text{for $\xi\ge\frac{t^2}{8}+1$},
\end{cases} 
\end{align*}
with the property that $y+H=\id$.
Finally, we find
\begin{align*}
l(t,\eta)&=\begin{cases}
-\infty, & \text{for $\eta=0$}, \\
\eta, & \text{for $\eta\in(0,1]$}, \\
\end{cases} \\[2mm]
\chi(t,\eta)&=\bar y(t, l(t,\eta))=\frac{t^2}{4}\left(\eta-\frac12\right), \quad \eta\in(0,1],  \\
\U(t,\eta)&= \bar U(t, l(t,\eta))=\frac{t}{2}\left(\eta-\frac12\right), \quad \eta\in(0,1],
\end{align*}
as expected.
\end{example}

\begin{theorem}\label{thm:existence}
Let $u_0\in H^1(\Real)$ and $\mu_0$ be a nonnegative, finite Radon measure with $C=\mu_0(\Real)$. Let $(u(t),\mu(t))$ denote the conservative solution of the Hunter--Saxton equation.  Define
\begin{align*}
\chi_0(\eta)&=\sup\{x\mid \mu_0((-\infty,x))<\eta\}, \\
\U_0(\eta)&=u_0(\chi_0(\eta)).
\end{align*}
If $\lim_{\eta\to 0}\chi_0(\eta)$ and $\lim_{\eta\to C}\chi_0(\eta)$ are finite, we define
\begin{align*}
\chi(t,\eta)&=
\begin{cases}
-\frac{C}{8}t^2+t\,\U_0(0)+\chi_0(0)+\eta, &\text{if $\eta<0$}, \\
\frac{t^2}{4}(\eta-\frac{C}2)+t\,\U_0(\eta)+\chi_0(\eta), & \text{if $\eta\in(0,C]$}, \\  
\frac{C}{8}t^2+t\U_0(C)+\chi_0(C)+\eta-C, &\text{if $\eta>C$},
\end{cases} \\[2mm]
\U(t,\xi)&=\begin{cases}
-\frac{C}{4}t+\U_0(0), &\text{if $\eta<0$}, \\
\frac{t}{2}(\eta-\frac{C}2)+\U_0(\eta), & \text{if  $\eta\in(0,C]$}, \\
\frac{C}{4}t+\U_0(C), &\text{if $\eta\ge C$}.
\end{cases}
\end{align*}
Then we have
\begin{multline*}
\{(x,t,u(t,x))\in \Real^3\mid t\in[0,\infty),\, x\in\Real \} \\
   =\{(\chi(t,\eta),t,\U(t,\eta))\in \Real^3\mid t\in[0,\infty),\, \eta\in\Real \}.
\end{multline*}

If $\lim_{\eta\to 0}\chi_0(\eta)=-\lim_{\eta\to C}\chi_0(\eta)=-\infty$, we define
\begin{align*}
\chi(t,\eta)&=
\frac{t^2}{4}(\eta-\frac{C}2)+t\,\U_0(\eta)+\chi_0(\eta), \quad \text{if $\eta\in(0,C)$}, \\  
\U(t,\xi)&=
\frac{t}{2}(\eta-\frac{C}2)+\U_0(\eta), \quad \text{if  $\eta\in(0,C)$}.
\end{align*}
Then we have
\begin{multline*}
\{(x,t,u(t,x))\in \Real^3\mid t\in[0,\infty),\, x\in\Real \} \\
   =\{(\chi(t,\eta),t,\U(t,\eta))\in \Real^3\mid t\in[0,\infty),\, \eta\in(0,C) \}.
\end{multline*}
Similar results hold if one of $\lim_{\eta\to 0}\chi_0(\eta)$ and $\lim_{\eta\to C}\chi_0(\eta)$ is finite. 
\end{theorem}

We can now introduce the new Lipschitz metric. Define
\begin{equation*}
d((u_1(t),\mu_1(t)) ,(u_2(t),\mu_2(t)))
= \norm{\U_1(t)-\U_2(t)}_{L^\infty([0,C])}+\norm{\chi_1(t)-\chi_2(t)}_{L^1([0,C])},
\end{equation*}
which implies that
\begin{align*}
d((u_1(t)&,\mu_1(t)) ,(u_2(t),\mu_2(t)))\\
& = \norm{\U_1(t)-\U_2(t)}_{L^\infty([0,C])}+\norm{\chi_1(t)-\chi_2(t)}_{L^1([0,C])}\\
& \leq (1+Ct)\norm{\U_1(0)-\U_2(0)}_{L^\infty([0,C])}+\norm{\chi_1(0)-\chi_2(0)}_{L^1([0,C])}\\
& \leq (1+Ct)d((u_1(0),\mu_1(0)),(u_2(0),\mu_2(0))).
\end{align*}

A drawback of the above construction is the fact that we are only able to compare solutions $(u_1, \mu_1)$ and $(u_2,\mu_2)$ with the same energy, viz.~$\mu_1(\Real)=\mu_2(\Real)=C$. The rest of this section is therefore devoted to overcoming this limitation.

A closer look at the system \eqref{sys:ode:metric} of ordinary differential equations  reveals that we can rescale $\chi(t,\eta)$ and $\U(t,\eta)$ in the following way. Let
\begin{align*}
\hat \chi(t,\eta)&= \chi(t, C\eta),\\
\hat \U(t,\eta)&= \U(t, C\eta),
\end{align*}
which  also covers the case $\mu(\Real)=0$, (which would correspond to the zero solution). Then $\hat\chi(t,\dott)\colon [0,1]\to\Real$, $\hat \U(t,\dott)\colon[0,1]\to\Real$ for all $t$ and 
\begin{align*}
\hat\chi_t(t,\eta)& = \hat \U(t,\eta),\\
\hat \U_t(t,\eta)& = \frac12C(\eta-\frac12).  
\end{align*}
Direct computations then yield
\begin{align*}
\norm{\hat \U_1(t,\dott)-\hat \U_2(t,\dott)}_{L^\infty([0,1])}
& \leq\norm{\hat \U_1(0,\dott)-\hat \U_2(0,\dott)}_{L^\infty([0,1])}+\frac12 t\vert C_1-C_2\vert \norm{\eta-\frac12}_{L^\infty([0,1])}\\
& \leq \norm{\hat \U_1(0,\dott)-\hat \U_2(0,\dott)}_{L\infty([0,1])}+\frac14 t \vert C_1-C_2\vert \notag
\end{align*}
and 
\begin{align*}
\norm{\hat\chi_1(t,\dott)-\hat \chi_2(t,\dott)}_{L^1([0,1])}
& \leq \norm{\hat \chi_1(0,\dott)-\hat\chi_2(0,\dott)}_{L^1([0,1])}+ t \norm{\hat \U_1(0,\dott)-\hat \U_2(0,\dott)}_{L^\infty([0,1])}\\
&\quad +\frac18 t^2\vert C_1-C_2\vert. 
\end{align*}
Thus introducing the redefined distance by 
\begin{align*}
d((u_1,\mu_1),(u_2,\mu_2))&=\norm{\chi_1(C_1\dott)-\chi_2(C_2\dott)}_{L^1([0,1])}\\
& \quad + \norm{u_1(\chi_1(C_1\dott))-u_2(\chi_2(C_2\dott))}_{L^\infty([0,1])} \notag\\
& \quad + \vert C_1-C_2\vert, \notag
\end{align*}
we end up with 
\begin{align*}
d((u_1(t),\mu_1(t))&,(u_2(t), \mu_2(t)))\\
&\leq  \left(1+t+\frac18 t^2\right)d((u_1(0),\mu_1(0)),(u_2(0), \mu_2(0))). \notag
\end{align*}
In particular, $ d((u_1,\mu_1),(u_2,\mu_2))=0$ immediately implies that $C_1=\mu_1(\Real)=\mu_2(\Real)=C_2$, which then implies $\chi_1(t,x)=\chi_2(t,x)$ and thus $u_1(t,x)=u_2(t,x)$. 

\begin{example}
Recall Example \ref{rem:soliton}.  Consider initial data $u_0=0$ and $\mu_{0,i}=\alpha_i\delta_0$ yielding solutions $(u_i(t),\mu_i(t))$.
Here we find
\begin{equation*}
d((u_0,\mu_{0,1}),(u_0, \mu_{0,2}))= d((0,\alpha_1\delta_0),(0, \alpha_2\delta_0))=\abs{\alpha_1-\alpha_2}.
\end{equation*}
Thus
\begin{equation*}
d((u_1(t),\mu_1(t)),(u_2(t), \mu_2(t)))\leq  \left(1+t+\frac18 t^2\right)\abs{\alpha_1-\alpha_2}.
\end{equation*}
\end{example}

\begin{theorem}\label{thm:main}
Consider $u_{0,j}$ and $\mu_{0,j}$ as in Theorem  \ref{thm:existence} for $j=1,2$ with $C_j=\mu_{0,j}(\Real)$.  Assume in addition that
\begin{equation}\label{eq:cond1A}
\int_{-\infty}^0 F_{0,j}(x)dx +\int_0^\infty (C_j-F_{j,0}(x))dx <\infty, \quad j=1,2.
\end{equation} 

Define the metric
\begin{align*}
d((u_1(t),\mu_1(t)), &(u_2(t),\mu_2(t))) \\
&\qquad = \norm{\U_1(t,C_1\dott)-\U_2(t,C_2\dott)}_{L^\infty([0,1])}\\
&\qquad\quad+\norm{\chi_1(t,C_1\dott)-\chi_2(t,C_2\dott)}_{L^1([0,1])}
+\abs{C_1-C_2}.
\end{align*}
Then we have
\begin{equation*}
d((u_1(t),\mu_1(t)) ,(u_2(t),\mu_2(t)))
 \leq \left(1+t+\frac18 t^2\right)d((u_{0,1},\mu_{0,1}),(u_{0,2},\mu_{0,2})).
\end{equation*}
\end{theorem}
\begin{proof}
It is left to show that $\chi_j(t)\in L^1([0,C_j])$ for all $t$ positive and $j=1,2$. For the remainder of this proof, fix $j$ and drop it in the notation. 

Note that \eqref{eq:cond1A} is equivalent to $\chi(0,\eta)\in L^1([0,C])$. Indeed, denote by $\eta_1$ the point at which $\chi(0,\eta)$ changes from negative to positive, then by definition
\begin{equation*}
\chi(0,\eta)=\sup\{x\mid F(0,x)<\eta\},
\end{equation*}
and thus 
\begin{align*}
\norm{\chi}_{L^1([0,C])}&= \int_{0}^C \vert \chi(\eta)\vert d\eta =-\int_0^{\eta_{1}}\chi(\eta)d\eta+\int_{\eta_{1}}^C \chi(\eta)d\eta\\
&= \int_{-\infty}^0 F(0,x) dx+\int_0^\infty (C-F(x))dx.
\end{align*}

It remains to show that $\chi(t)$  remains integrable, i.e., $\chi(t)\in L^1([0,C])$ for all $t$ positive. Translating the condition \eqref{eq:cond1A} we find
\begin{equation}\label{eq:cond2a}
\int_{-\infty}^{\xi_1} \bar H\bar y_\xi(\xi)d\xi+\int_{\xi_1}^\infty (C-\bar H)\bar y_\xi(\xi)d\xi <\infty,
\end{equation}
where $\xi_1$ is chosen such that $\bar y(\xi_1)=0$. Note that  it does not matter if there exists a single point or a whole interval such that $\bar y(\xi)=0$, since in the latter case $\bar y_\xi(\xi)=0$.
Denote by $\xi(t)$ the time dependent function such that $\bar y(t,\xi(t))=0$ for all $t$, which is not unique. 
Then the first term can be rewritten as 
\begin{align*}
\int_{-\infty}^{\xi(t)} & \bar H(t,\xi)\bar y_\xi(t,\xi)d\xi = \int_{-\infty}^{\xi(t)} \bar H(0,\xi)\bar y_\xi(t,\xi)d\xi\\
& = \int_{-\infty}^{\xi(t)}\bar H(0,\xi)(\bar y_\xi(0,\xi)+t\bar U_\xi(0,\xi)+\frac14 t^2 \bar H_\xi(0,\xi))d\xi,\\
& \leq \int_{-\infty}^{\xi(t)}(1+t)\bar H\bar y_\xi(0,\xi)d\xi +\int_{-\infty}^{\xi(t)}(t+\frac14 t^2) \bar H\bar H_\xi(0,\xi)d\xi\\
& = (1+t) \int_{-\infty}^{\xi(0)} \bar H\bar y_\xi(0,\xi) d\xi+ (1+t) \int_{\xi(0)}^{\xi(t)}\bar H \bar y_\xi(0,\xi)d\xi+ \left(\frac12 t+\frac18 t^2\right) H^2(0,\xi(t))\\
& \leq (1+t)\int_{-\infty}^{\xi(0)} \bar H\bar y_\xi(0,\xi)d\xi +(1+t)\int_{\xi(0)}^{\xi(t)}\bar H \bar y_\xi(0,\xi)d\xi+ \left(\frac12 t+\frac18 t^2\right) C^2,
\end{align*}
where we used \eqref{eq:tildenull} and that $\bar U^2_\xi(t,\xi)=\bar H_\xi \bar y_\xi(t,\xi)$. The term on the right hand side will be finite if we can show that the second integral on the right hand side is finite. Therefore observe that 
\begin{align*}
\int_{\xi(0)}^{\xi(t)} \bar H \bar y_{\xi}(0,\xi) d\xi& = \int_{\bar y(0,\xi(0))}^{\bar y(0,\xi(t))} F(x) dx = \int_0^{\bar y(0,\xi(t))} F(x) dx = \int_{\bar y(t,\xi(t))}^{\bar y(0,\xi(t))} F(x)dx \\
& \leq C\vert y(t,\xi(t))-\bar y(0,\xi(t))\vert \leq C( t \vert U(0,\xi(t))\vert + \frac{t^2}{8} C)\\
& \leq C(t \norm{u_0}_{L^\infty(\Real)}+\frac{t^2}{8} C)<\infty.
\end{align*}

Similar considerations yield that the second integral in \eqref{eq:cond2a} remains finite as time evolves.

\end{proof}
\begin{remark}
Observe that the distance introduced in Theorem \ref{thm:main} gives at most a quadratic growth in time, while the distance in \cite{BHR} has at most an exponential growth in time.
\end{remark}
We make a comparison with the more complicated  Camassa--Holm equation in the next remark.
\begin{remark}
Consider an interval $[\xi_1,\xi_2]$ such that $U_0(\xi)=U_0(\xi_1)$ and $H(\xi_1)=H(\xi)$ for all $\xi\in[\xi_1,\xi_2]$. This property will  remain true for all later times. In particular, this means that these intervals do not show up in our metric, 
and the function $\chi(t,\eta)$ always has a constant jump at the corresponding point $\eta$. This is in big contrast to the Camassa--Holm equation where jumps in $\chi(t,\eta)$ may be created and then subsequently disappear immediately again. Thus the construction for the Camassa--Holm equation is much more involved than the HS construction.
\end{remark}
This is illustrated in the next examples. 

\begin{example}
Given the initial data $(u_0,\mu_0)=(0, \delta_0+2\delta_1)$, direct calculations yield 
\begin{align*}
u(t,x)&=\begin{cases}
-\frac34 t, & x\leq -\frac38 t^2,\\
\frac2{t} x, & -\frac38 t^2\leq x\leq -\frac18 t^2,\\
-\frac14 t, & -\frac18 t^2\leq x\leq 1-\frac18 t^2,\\
\frac2{t} (x-1), & 1-\frac18 t^2\leq x\leq 1+\frac38 t^2,\\
\frac34 t, & 1+\frac38 t^2\leq x,
\end{cases}\\
F(t,x) & = \begin{cases}
 0, & x\leq -\frac38 t^2,\\
 \frac32 +\frac4{t^2}x, & -\frac38 t^2\leq x\leq -\frac18 t^2,\\
 1, & -\frac18 t^2\leq x\leq 1-\frac18 t^2,\\
 \frac32 +\frac4{t^2} (x-1), & 1-\frac18  t^2\leq x\leq 1+\frac38 t^2,\\
 3, & 1+\frac38 t^2\leq x.
 \end{cases}
 \end{align*}
 Calculating the pseudo inverse $\chi(t,\eta)$ and $\U(t,\eta)=u(t,\chi(t,\eta))$ for each $t$, then yields
 \begin{align*}
\chi(t,\eta)& = \begin{cases}
-\infty, & \eta=0,\\
\frac{t^2}4(\eta-\frac32), & 0<\eta\leq 1,\\
1+\frac{t^2}{4}(\eta-\frac32), & 1<\eta\leq 3,
\end{cases}\\
\U(t,\eta)& = 
\frac{t}2 (n-\frac32),\quad \text{ for all } \eta\in[0,3].
\end{align*} 
Here two observations are important. Note that $\U(t,\eta)$ is continuous and differentiable with respect to $\eta$, while $\chi(t,\eta)$ on the other hand has at each time $t$ a discontinuity at $\eta=1$ (and of course at $\eta=0$). In particular, one has 
\begin{equation*}
\lim_{\eta\to 1-}\chi(t,\eta)=-\frac18 t^2\quad \text{ and } \quad \lim_{\eta\to 1+} \chi(t,\eta)=1-\frac18 t^2.
\end{equation*}
Thus the jump in function value remains unchanged even if the limit from the left and the right are time dependent.

In order to understand the behavior of $l(t,\eta)$, let us have a look at the solution in Lagrangian coordinates, which is given by 
\begin{align*}
y(t,\xi) &= \begin{cases}
\xi-\frac38 t^2, & \quad \xi\leq 0,\\
\frac14 (\xi-\frac32)t^2, & \quad 0\leq \xi\leq 1,\\
\xi-1-\frac18 t^2, & \quad 1\leq \xi\leq 2,\\
1+\frac14 (\xi-\frac52)t^2, & \quad 2\leq \xi\leq 4,\\
\xi-3+\frac38 t^2, & \quad 4\leq\xi,
\end{cases} \\
U(t,\xi) &= \begin{cases}
-\frac34 t, & \quad \xi\leq 0,\\
\frac12 (\xi-\frac32)t, & \quad 0\leq \xi\leq 1,\\
-\frac14 t, & \quad 1\leq\xi\leq 2,\\
\frac12 (\xi-\frac52)t,& \quad  2\leq\xi\leq 4,\\
\frac34 t, & \quad 4\leq \xi,
\end{cases} \\
H(t,\xi)&=\begin{cases}
0, & \quad \xi\leq 0,\\
\xi, & \quad 0\leq \xi\leq 1,\\
1, & \quad 1\leq \xi\leq 2,\\
\xi-1, & \quad 2\leq\xi\leq 4,\\
3, & \quad 4\leq \xi.
\end{cases}
\end{align*}
Hence direct computations yield that 
\begin{equation*}
l(t,\eta)=l(0,\eta)=\begin{cases}
-\infty, & \quad \eta=0,\\
\eta, &\quad 0<\eta \leq1,\\
\eta+1,&\quad 1< \eta \leq 3 .
\end{cases}
\end{equation*}
Note that $l(0,\eta)$ is non-constant on the intervals where both $\U(0,\eta)$ and $\chi(0,\eta)$ are constant.

\end{example}



\begin{thebibliography}{99}


\bibitem{MR2191785}
A.~Bressan, A.~Constantin, Global solutions of the {H}unter--{S}axton equation,
  SIAM J. Math. Anal. 37~(3) (2005) 996--1026 (electronic).

\bibitem{MR2278406}
A.~Bressan, A.~Constantin, Global conservative solutions of the
  {C}amassa--{H}olm equation, Arch. Ration. Mech. Anal. 183~(2) (2007)
  215--239.

\bibitem{MR2134955}
 F.~Bolley, Y.~Brenier, and G.~Loeper,
Contractive metrics for scalar conservation laws,
 J. Hyperbolic Differ. Equ., 2~(1) (2005) 91--107.

\bibitem{BHR}
A.~Bressan, H.~Holden, and X.~Raynaud,
Lipschitz metric for the Hunter--Saxton equation,
J. Math. Pures Appl. (9) 94 (2010) 68--92.  

\bibitem{CDG06}
J.~A. Carrillo, M.~Di~Francesco, and M.~P. Gualdani,
Semidiscretization and long-time asymptotics of nonlinear
diffusion equations, Comm. Math. Sci. 5 (2007), 21--53.

\bibitem{CDL}
J.~A. Carrillo, M.~Di~Francesco, and C. Lattanzio,
Contractivity of Wasserstein metrics and asymptotic profiles
for scalar conservation laws, J. Differential Equations 231 (2006), 425--458.

\bibitem{CGT03}
J.~A. Carrillo, M.~P. Gualdani, and G.~Toscani, Finite speed
of propagation in porous media by mass transportation methods, C.
R. Acad. Sci. Paris  338 (2004), 815--818.

\bibitem{CT03}
J.~A. Carrillo and G.~Toscani, Wasserstein metric and
large-time asymptotics of nonlinear diffusion equations, In: {\em New
Trends in Mathematical Physics}, World Sci., Publ., Hackensack, NJ,   (2004), 234--244.

\bibitem{CTPE}
J. A. Carrillo, and G. Toscani, Contractive probability metrics and asymptotic behavior of
dissipative kinetic equations, Rivista Matem\`atica di Parma 6 (2007), 75--198.

\bibitem{HolKarRis:sub05}
H.~Holden, K.~H. Karlsen, N.~H. Risebro, Convergent difference schemes for the
  Hunter--Saxton equation, Math. Comp., 76 (2007) 699--744.

\bibitem{HolRay:07}
H.~Holden, X.~Raynaud, Global conservative solutions of the {C}amassa--{H}olm
  equation---a {L}agrangian point of view, Comm. Partial Differential Equations
  32~(10-12) (2007) 1511--1549.

\bibitem{MR1135995}
J.~K. Hunter, R.~Saxton, Dynamics of director fields, SIAM J. Appl. Math.
  51~(6) (1991) 1498--1521.


\bibitem{MR1361013}
J.~K. Hunter, Y.~X. Zheng, On a nonlinear hyperbolic variational equation. {I}.
  {G}lobal existence of weak solutions, Arch. Rational Mech. Anal. 129~(4)
  (1995) 305--353.

\bibitem{MR1361014}
J.~K. Hunter, Y.~X. Zheng, On a nonlinear hyperbolic variational equation.
  {II}. {T}he zero-viscosity and dispersion limits, Arch. Rational Mech. Anal.
  129~(4) (1995) 355--383.
  
\bibitem{LT}
H. Li, and G. Toscani, Long--time asymptotics of kinetic  models of granular flows, Arch. Ration. Mech. Anal. 172 (3) (2004) 407--428.

\bibitem{anders}
 A.~Nordli, 
 A Lipschitz metric for conservative solutions of the two-component Hunter--Saxton system, 
Methods Appl. Anal. 23 (2016) 215--232.


\bibitem{andersPhD}
 A.~Nordli,  On the Two-Component Hunter--Saxton System, 
PhD dissertation, NTNU 2017.

\bibitem{Vil03}
C.~Villani, \emph{Topics in Optimal Transportation}, 
Amer. Math. Soc., Providence, RI, 2003.

\bibitem{MR1668954}
P.~Zhang, Y.~Zheng, On oscillations of an asymptotic equation of a nonlinear
  variational wave equation, Asymptot. Anal. 18~(3-4) (1998) 307--327.

\bibitem{MR1701136}
P.~Zhang, Y.~Zheng, On the existence and uniqueness of solutions to an
  asymptotic equation of a variational wave equation, Acta Math. Sin. (Engl.
  Ser.) 15~(1) (1999) 115--130.

\bibitem{MR1799274}
P.~Zhang, Y.~Zheng, Existence and uniqueness of solutions of an asymptotic
  equation arising from a variational wave equation with general data, Arch.
  Ration. Mech. Anal. 155~(1) (2000) 49--83.
 
\end{thebibliography}
\end{document}